\documentclass{amsart}
\usepackage[all]{xy}
\usepackage[pdftex]{graphicx}
\usepackage{amsmath,amssymb}
\usepackage{amsthm}
\usepackage{url}

\newtheorem{prop}{Proposition}[section]
\newtheorem{lemma}[prop]{Lemma}
\newtheorem{them}[prop]{Theorem}

\theoremstyle{definition}
\newtheorem{defini}[prop]{Definition}
\newtheorem{example}[prop]{Example}
\newtheorem{remark}[prop]{Remark}

\allowdisplaybreaks[4]

\newcommand{\ep}{\varepsilon}
\newcommand{\RR}{\mathbb R}
\newcommand{\CC}{\mathbb C}
\newcommand{\ZZ}{\mathbb Z}

\newcommand{\grad}{{\rm grad\hskip0.3mm}}
\newcommand{\ind}{{\rm ind}}
\newcommand{\Crit}{{\rm Crit}}
\newcommand{\UL}{\overrightarrow{ab}}
\newcommand{\LU}{\overrightarrow{ba}}
\newcommand{\UU}{\overrightarrow{aa}}
\newcommand{\LL}{\overrightarrow{bb}}

\title{
Gluing formula for an invariant related to
the Chern-Simons perturbation theory}
\author{
Teruaki Kitano
}

\address{Department of Information Systems Science, 
Faculty of Science and Engineering, 
Soka University, 
Tangi-cho 1-236, 
Hachioji, Tokyo 192-8577, Japan}

\email{kitano@soka.ac.jp}
\author{
Tatsuro Shimizu
}

\address{
School of System Design and Technology,
Tokyo Denki University
, 
5 Senju Asahi-cho, Adachi-ku, Tokyo 120-8551, Japan}

\email{shimizu@mail.dendai.ac.jp}


\begin{document}
\maketitle

\bigskip
\begin{abstract}
The Chern-Simons perturbation theory gives an invariant $d(M,\rho)$ for a pair of a closed oriented 3-manifold $M$ and a representation $\rho$ of the fundamental group.
We generalize $d(M,\rho)$ for compact oriented 3-manifolds with torus boundaries. 
Further we give a gluing formula for $d(M,\rho)$.

\end{abstract}

\large


\section{Introduction}
The Chern-Simons perturbation theory was established by M.~Kontsevich \cite{Kon}, S.~Axelrod and I.~M.~Singer \cite{AS}.
It gives invariants for a pair $(M,\rho)$ of a closed oriented 3-manifold $M$ and a representation $\rho$ of the fundamental group. 
Here a representation $\rho$ is assumed to be acyclic, that is,
all homology groups with the local coefficient system given by $\rho$ are vanishing.
In the Chern-Simons perturbation theory, invariants are formulated
by using the configuration space integral.
For a Jacobi diagram $\Gamma$, the configuration space integral gives   
a homology class $I_{\Gamma}(M,\rho)$.
Each $I_{\Gamma}(M,\rho)$ depends on a choice of a propagator which is the integrand of the configuration space integral. However, certain weighted sums of $I_{\Gamma}(M,\rho)$ for Jacobi diagrams give invariants of $(M,\rho)$.
Here is a sketch of the construction of the Chern-Simons perturbation theory.

The configuration space integral makes sense even for graphs rather than Jacobi diagrams. 
The invariant discussed in this article can be considered as
the configuration space integral $I_{\circ}(M,\rho)$ for the graph $\bigcirc$ with only one vertex and one edge.
Remark that the graph $\bigcirc$ is not a Jacobi diagram.
However $\bigcirc$ often appears in many Jacobi diagrams as a sub-graph.
The configuration space integral $I_{\circ}(M,\rho)$ can be considered as a primitive part of the Chern-Simons perturbation theory.

A homology class $I_{\circ}(M,\rho)$ is in the 1-dimensional homology group 
$H_1(M;V_{\rho^*}\otimes V_{\rho})$ of $M$ with local coefficients.
Here $V_{\rho}$ and $V_{\rho^*}$ are local systems on $M$ given by $\rho$ and $\rho^*$, which is the dual representation of $\rho$.
It can be seen that $I_{\circ}(M,\rho)$ depends on a choice of a propagator, but
this ambiguity of $I_{\circ}(M,\rho)$ is only in $H_1(M,\ZZ)$. 
Then $I_{\circ}(M,\rho)$ modulo $H_1(M;\ZZ)$ is an invariant of a pair $(M,\rho)$:
$$d(M,\rho)=I_{\circ}(M,\rho)~~\text{mod}~~ H_1(M;\ZZ)
\in H_1(M;V_{\rho}^*\otimes V_{\rho})/H_1(M;\ZZ).$$

The invariant $d(M,\rho)$ was studied as an invariant first by C. Lescop \cite{Lescop1} for $M$ with the first Betti number $b_1(M)=1$ and the abelian representation. It is denoted by $I_{\Delta}(M)$ in \cite{Lescop1}.
A. Cattaneo and the second author also studied $d(M,\rho)$ for $M$ with $b_1(M)>1$ and an abelian representation $\rho$ \cite{ShimizuJDG}, \cite{CS}.
We note that the invariant $d(M,\rho)$ plays a role as a kind of a defect in the Chern-Simons perturbation theory
(see Remark 2.4 of \cite{CS} for example).

It is expected that the invariant $d(M,\rho)$ and the Reidemeister torsion ${\rm Tor}(M,\rho)$ of $\rho$ are essentially equivalent.  
Actually, for an acyclic abelian representation $\rho$ on $M$ with $b_1(M)=1$, the equivalence was proved by Lescop \cite{Lescop1}, M.~Hatchings and Y.~J.~Lee \cite{HL}, \cite{HL2}, \cite{HL3}.
After that it was proved 
for $b_1(M)>1$ by the second author \cite{Shimizu}.
More precisely, $d(M,\rho)$ can be computed from the Reidemeister torsion by using a variant of logarithmic derivatives in the above cases. 
We note that
the origin of the Chern-Simons perturbation theory is Witten's proposal on the partition function of the Chern-Simons quantum field theory  \cite{Witten}.
It is expected that invariants coming from the partition function of the Chern-Simons quantum field theory have an information of the Reidemeister torsion. 

In this article, we generalize the invariant $d(M,\rho)$ for compact oriented 3-manifolds $M$ such that the boundary of $M$ is a disjoint union of tori by using a Morse theoretic description.
Further we obtain a gluing formula for $d(M,\rho)$ as follows:
Let $N_a, N_b$ be compact oriented 3-manifolds with torus boundaries
and let $\rho_a,\rho_b$ be representations of the fundamental groups of $N_a,N_b$ respectively.
Let $M=N_a\cup N_b$ be a closed manifold given by gluing $N_a$ and $N_b$ along the boundary.
We assume that there is a representation $\rho$ of the fundamental group of $M$ so that $\rho|_{N_a}=\rho_a$ and $\rho|_{N_b}=\rho_b$.
The representations $\rho_a,\rho_b$ and $\rho$ are assumed to be acyclic.
Then it follows from that the representation $\rho|_{\partial N_a}$ of the fundamental group of $\partial N_a$ is also acyclic from the short exact sequence of $H_*(N_a;V_{\rho_a}), H_*(N_b;V_{\rho_b})$ and $H_*(M;V_{\rho})$.
In this situation, a gluing formula among
$d(M,\rho),d(N_a,\rho_a)$ and $d(N_b,\rho_b)$ holds:
\begin{them}
\begin{eqnarray}
d(M,\rho)=
(\iota_a)_*d(N_a,\rho_a)+(\iota_b)_*d(N_b,\rho_b).
\end{eqnarray}
\end{them}
Here $(\iota_a)_*$ and $(\iota_b)_*$ are induced by the inclusion $\iota_a:N_a\to M$ and $\iota_b:N_b\to M$ respectively.
This formula can be considered as an logarithmic analogy of the following well-known gluing formula for the Reidemeister torsion (for example, see \cite{Johnson}):
$${\rm Tor}(M,\rho)={\rm Tor}(N_a,\rho_a){\rm Tor}(N_b,\rho_b)$$ 
This suggests that $d(M,\rho)$ can be computed from the Reidemeister torsion via a variant of a logarithm as in the case for abelian representations.

The contents of this paper are as follows.
In Section~2 we review the definition of $d(M,\rho)$ of closed manifolds.
In Section~3 we give a Morse theoretic description of $d(M,\rho)$.
By using this description, $d(M,\rho)$ can be generalized for manifolds with torus boundaries.
A gluing formula for this generalized $d(M,\rho)$ is given in Section~4.
In Section 5, the well-definedness of the generalized $d(M,\rho)$ and the gluing formula stated in Section~4 are proved.
\subsection*{Acknowledgments}
This work was (partly) supported by JSPS KAKENHI Grant Number JP19K03505 and JP18K13408.
\section{Review of $d(M,\rho)$ of a closed manifold $M$}
In this section, we review a definition of $d(M,\rho)$.
See \cite{Shimizu} for more details.

Let $M$ be a closed oriented 3-manifold.
Let $\rho:\pi_1(M)\to {\rm GL}(V)$ be a representation on a vector space $V$ over a field $F$.
The representation $\rho$ gives a local system $V_{\rho}$ on $M$ as follows. A vector space $V_x\cong V$ is assigned to each $x\in M$.
For a path $\gamma$ from $x\in M$ to $y\in M$, an isomorphism which is denoted by $\gamma_*:V_x\stackrel{\cong}{\to} V_y$ is assigned.
A closed loop $\gamma$ gives an element $[\gamma]$ of $\pi_1(M)$.
Such a $\gamma$ satisfies $\gamma_*=\rho([\gamma])$.
We assume that $\rho$ is acyclic.
A representation $\rho$ is said to be acyclic if the homology groups with the corresponding local system coefficient are vanishing: $H_*(M;V_{\rho})=0$.
We denote by $V_{\rho^*}$ the local system given by the dual representation $\rho^*$ of $\rho$.
Let $\pi_i:M\times M\ni (x_1,x_2)\mapsto x_i\in M$ the projection for $i=1,2$.
By the K\"{u}nneth formula, 
a local system $V_{\rho^*}\boxtimes V_{\rho}=\pi_1^*V_{\rho^*}\otimes \pi_2^*V_{\rho}$ on $M\times M$ is acyclic.

Let $\Delta=\{(x,x)\mid x\in M\}\subset M\times M$ be the diagonal.
The restriction of $V_{\rho^*}\boxtimes V_{\rho}$ to 
$\Delta$ is equal to $V_{\rho^*}\otimes V_{\rho}={\rm Hom}(V,V)_{\rho^*\otimes \rho}={\rm Hom}(V,V)_{\rho}$. 
The identity homomorphism 
${\bf 1}\in {\rm Hom}(V,V)=V^*\otimes V$
gives the canonical section of $V_{\rho^*}\otimes V_{\rho}$.
Then, for a chain $c\in C_*(\Delta;\ZZ)$, we have a chain $c{\bf 1}$ of 
$C_*(\Delta;V_{\rho^*}\otimes V_{\rho})$.
Thus $\Delta$ gives a 3-cycle
$$\Delta{\bf 1}\in C_3(\Delta;V_{\rho^*}\otimes V_{\rho}
)\subset C_3(M\times M;V_{\rho^*}\boxtimes V_{\rho}).$$
Let $N(\Delta)\subset M\times M$ be a closed tubular neighborhood of $\Delta$. 
$N(\Delta)$ is a 3-dimensional closed disk $B^3$-bundle over $\Delta$. 
Since $\partial N(\Delta)\cong \Delta\times \partial B^3$, there is
a section $s:\Delta\to \partial N(\Delta)$ of $\partial B^3$-bundle $\partial N(\Delta)\to \Delta$. 
Then $s(\Delta){\bf 1}$ is a 3-cycle of $C_3(M\times M;V_{\rho^*}\boxtimes V_{\rho})$. 
Since $H_3(M\times M;V_{\rho^*}\boxtimes V_{\rho})=0$,
there exists a 4-chain $\Sigma\in C_4(M\times M;V_{\rho^*}\boxtimes V_{\rho})$ satisfying $\partial \Sigma=s(\Delta){\bf 1}$. 
Roughly speaking, $d(M,\rho)$ is defined to be the intersection of $\Sigma$ and $\Delta$.
For a precise definition, we take some homomorphisms $\partial_*, E$ and $\iota$.
The interior of $N(\Delta)$ is denoted by $\mathring{N}(\Delta)$.
First
$$\partial_*:H_4(M\times M,M\times M\setminus \mathring{N}(\Delta);V_{\rho^*}\boxtimes V_{\rho})\to H_3(M\times M\setminus \mathring{N}(\Delta);V_{\rho^*}\boxtimes V_{\rho})$$
is the connecting homomorphism of the homology long exact sequence for $(M\times M, M\times M\setminus \mathring{N}(\Delta))$. Since $V_{\rho^*}\boxtimes V_{\rho}$ is acyclic, $\partial_*$ is an isomorphism.
Secondly let 
$$E:H_4(M\times M,M\times M\setminus\mathring{N}(\Delta);V_{\rho}\boxtimes V_{\rho^*})
\stackrel{\cong}{\to} H_4(N(\Delta),\partial N(\Delta);V_{\rho^*}\boxtimes V_{\rho})$$
be the excision isomorphism.

Recall that $V_{\rho^*}\boxtimes V_{\rho}\cong V_{\rho^*}\otimes V_{\rho}$ on $\Delta\subset M\times M$.
Then $H_1(M;\ZZ)$ is considered as a subspace of $H_1(M;V_{\rho^*}\otimes V_{\rho})$ via $H_1(M;\ZZ)\cong H_1(M;\ZZ){\bf 1}\hookrightarrow H_1(M;V_{\rho^*}\otimes V_{\rho}).$
Then we can consider the Thom isomorphism 
$$\iota:H_4(N(\Delta),\partial N(\Delta);V_{\rho^*}\boxtimes V_{\rho})
\stackrel{\cong}{\to} H_1(\Delta;V_{\rho^*}\otimes V_{\rho}).$$
Finally we obtain an isomorphism
$$\Phi=\iota\circ E\circ \partial_*^{-1}:H_3(M\times M\setminus\mathring{N}(\Delta);V_{\rho^*}\boxtimes V_{\rho})\stackrel{\cong}{\to} 
H_1(\Delta;V_{\rho^*}\otimes V_{\rho})\cong H_1(M;V_{\rho^*}\otimes V_{\rho})$$ 
by taking the composition of the aboves and the canonical diffeomorphism $\Delta\ni (x,x) \mapsto x\in M$.
The 3-cycle $s(\Delta){\bf 1}$ gives a homology class in $H_3(M\times M\setminus \mathring{N}(\Delta);V_{\rho^*}\boxtimes V_{\rho})$.
Thus we have a homology class 
$\Phi [s(\Delta){\bf 1}]\in H_1(M;V_{\rho^*}\otimes V_{\rho})$.
It depends on a choice of a section $s:\Delta\to\partial N(\Delta)$.
Its equivalence class in $H_1(M;V_{\rho^*}\otimes V_{\rho})/H_1(M;\ZZ){\bf 1}$ is, however, independent of a choice of $s$:
\begin{lemma}
For any sections $s_1,s_2:\Delta\to \partial N(\Delta)$,
$$\Phi[s_1(\Delta){\bf 1}]-\Phi[s_2(\Delta){\bf 1}]\in H_1(M;\ZZ){\bf 1}.$$ 
\end{lemma}
\begin{proof}
Fix a trivialization $N(\Delta)\cong\Delta\times B^3$.
Let ${\rm pt}_{\partial B^3}$ be a point of $\partial B^3$
and ${\rm pt}_{\Delta}$ a point of $\Delta$.
Then 
\begin{eqnarray*}
H_3(\partial N(\Delta);\ZZ)
&=&\left(H_1(\Delta;\ZZ)\otimes H_2(\partial B^3;\ZZ)\right)\oplus\left(H_3(\Delta;\ZZ)\otimes H_0(\partial B^3;\ZZ)\right)\\
&=&\left(H_1(\Delta;\ZZ)\otimes [\partial B^3]\right)\oplus\left(H_3(\Delta;\ZZ)\otimes [{\rm pt}_{\partial B^3}]\right).
\end{eqnarray*}
Here $H_1(\Delta;\ZZ)\otimes [\partial B^3]$ 
denotes the image of a homorphism 
$H_1(\Delta;\ZZ)\ni \alpha\mapsto \alpha\otimes [\partial B^3]\in H_1(\Delta;\ZZ)\otimes H_3(\partial B^3;\ZZ)$.
Because 
\begin{eqnarray*}
&&([s_1(\Delta)]-[s_2(\Delta)])\cap ([{\rm pt}_{\Delta}]\otimes [\partial B^3])\\
&=&[s_1(\Delta)]\cap ([{\rm pt}_{\Delta}]\otimes [\partial B^3])-[s_2(\Delta)]\cap ([{\rm pt}_{\Delta}]\otimes [\partial B^3])\\
&=&1-1\\
&=&0,
\end{eqnarray*}
it holds 
$[s_1(\Delta)]-[s_2(\Delta)]\in H_1(\Delta;\ZZ)\otimes [\partial B^3]
\subset H_1(\Delta;\ZZ)\otimes H_2(\partial B^3;\ZZ)$.
Thus there exists a homology class $\alpha\in H_1(\Delta;\ZZ)$ so that $\alpha\otimes[\partial B^3]
=[s_1(\Delta)]-[s_2(\Delta)]$.
Therefore
\begin{eqnarray*}
\Phi([s_1(\Delta){\bf 1}]-[s_2(\Delta){\bf 1}])
&=&\iota\circ E\circ\partial_*^{-1}(\alpha{\bf 1}\otimes [\partial B^3])\\
&=&\iota\circ E(\alpha{\bf 1}\otimes [B^3,\partial B^3])\\
&=&\iota(\alpha{\bf 1}\otimes [B^3,\partial B^3])\\
&=&\alpha{\bf 1}\in H_1(M;\ZZ){\bf 1}\subset H_1(M;V_{\rho^*}\otimes V_{\rho}).
\end{eqnarray*}
\end{proof}
By the above lemma, we can define $d(M,\rho)$ as follows:
\begin{defini}
$$d(M,\rho)=\Phi [s(\Delta){\bf 1}]\in H_1(M;V_{\rho^*}\otimes V_{\rho})/H_1(M;\ZZ).$$
\end{defini}

\section{Morse theoretical description of $d(M,\rho)$}
We give a Morse theoretical description of the defect $d(M,\rho)$ by using the Morse homotopy theory in this context established by K.~Fukaya \cite{Fukaya} and T.~Watanabe~\cite{Watanabe}. 
As a general reference for the Morse theory, see \cite{PM111} for example. 

Take a Morse function $f:M\to \RR$.
Let $\grad f$ be a gradient like vector field for $f$ satisfying the Morse-Smale condition.
Let $\{\Phi_t^f:M\to M\}_{t\in \RR}$ denote the 1-parameter family of diffeomorphisms associated to $-\grad f$.
Let ${\rm Crit}_i(f)$ denote the set of critical points of $f$ of Morse index $i$. 
Set ${\rm Crit}(f)=\sqcup_i{\rm Crit}_i(f)$.
For a critical point $p\in {\rm Crit}(f)$, we denote by ${\rm ind}(p)$ its Morse index.
For $p\in {\rm Crit}(f)$ and $q\in {\rm Crit}_{\ind(p)-1}(f)$, we denote by $\mathcal M(p,q)$ the set of trajectories connecting $p$ and $q$.
We orient the ascending manifold $\mathcal A_p$ of $p$ and the descending manifold $\mathcal D_p$ of $p$ by imposing the condition 
$T_p\mathcal A_p\oplus T_p\mathcal D_p=T_pM$ for each $p\in {\rm Crit}(f)$.
Each trajectory is a part of the intersection of an ascending manifold and a descending manifold. 
Then by using orientations of ascending manifold, descending manifold and $M$, each trajectory becomes an oriented 1-manifold. 
We assign $\ep(\gamma)=\pm 1$ to each $\gamma\in \mathcal M(p,q)$ as follows: if the orientation of $\gamma$ is from $p$ to $q$, then $\ep(\gamma)=1$; if the orientation of $\gamma$ is opposite, then $\ep(\gamma)=-1$.

For a singular 1-simplex $\sigma:[0,1]\to M$ and 
$a\in V_{\sigma(0)}^*\otimes V_{\sigma(0)}$,
we have a 1-chain 
$\sigma a\in C_1(M;V_{\rho^*}\otimes V_{\rho}).$
We consider $\gamma\in \mathcal M(p,q)$ as a singular 1-simplex. 
Thus for any $a\in V_p^*\otimes V_p$, we have a 1-chain 
$$\gamma a\in C_1(M;V_{\rho^*}\otimes V_{\rho}).$$
We assign $(-1)^{{\rm ind}(p)}$ to each critical point $p\in {\rm Crit}(f)$. Then ${\rm Crit}(f)$ becomes a $0$-chain,
denoted by $\widehat{\Crit}(f)$, of $C_0(M;\ZZ)$.
Since $\chi(M)=0$, the $0$-chain $\widehat{\Crit}(f)$ is a boundary
element in $C_0(M;\ZZ)$.
Now take a $1$-chain 
$$c_f\in C_1(M;\ZZ)$$
satisfying $\partial c_f={\rm Crit}(f)=\sum_{p}(-1)^{\ind (p)}p$. 
Then $c_f{\bf 1}$ is a 1-chain of $C_1(M;V_{\rho^*}\otimes V_{\rho})$.

With the above notations, the ($\rho$-twisted) Morse-Smale complex 
$(C_*,\partial_*)$
of $\grad f$
is written as follows:
\begin{eqnarray*}
C_i&=&\bigoplus_{p\in{\rm Crit}_i(f)}V_p(\cong V^{\sharp {\rm Crit}_i(f)}),\\
(\partial_i:C_i\to C_{i-1})&=&\sum_{p\in\Crit_i(f),q\in\Crit_{i-1}(f)}(\partial_{q,p}:V_p\to V_q),\\
(\partial_{q,p}:V_p\to V_q)&=&\sum_{\gamma\in {\mathcal M}(p,q)}(-1)^
{{\rm ind}(p)}\ep(\gamma)\gamma_*.
\end{eqnarray*}
\begin{defini}
A family of homomorphisms
$G_*=\{G_i:C_{i-1}\to C_i\}_i$ is said to be a {\it combinatorial propagator} for $(C_*,\partial_*)$ if $G_*$ satisfies the following conditions:
$$\partial_{i+1}\circ G_{i+1}+G_i\circ \partial_i={\rm id}_{C_i}~\text{for any}~i$$
\end{defini}
For a trajectory $\gamma\in \mathcal M(p,q)$,
a homomorphism $\gamma_*$ is from $V_{p}$ to $V_{q}$,
thus $\ep(\gamma)\gamma_*$ is also a homomorphism from $V_p$ to $V_q$.
Let $G_{p,q} 
:V_q\to V_p$ denote the part of $G_{i}:C_{i-1}\to C_i$ given by
$G_{p,q}=\pi_{V_p}\circ G_i|_{V_q}$, 
where $\displaystyle\pi_{V_p}:C_i=\oplus_{p'\in \Crit_i(f)}V_{p'}\to V_p$ is the projection.
Thus $G_{p,q}\circ \ep(\gamma)\gamma_*$ is a homomorphism from $V_p$ to itself.
By taking a tensor product with the identity homomorphism $1:V_p^*\to V_p^*$, we have a homomorphism 
$1\otimes G_{p,q}\circ \ep(\gamma)\gamma_*$ from 
$ V_p^*\otimes V_p$ to itself.
Recall that $\gamma$ is considered as a 1-chain. 
Therefore $\gamma(1\otimes G_{p,q}\circ \ep(\gamma)\gamma_*){\bf 1}$ is a 1-chain of $C_1(M;V_{\rho^*}\otimes V_{\rho})$.
Furthermore
$(1\otimes G_{p,q}\circ \ep(\gamma)\gamma_*){\bf 1}
=G_{p,q}\circ\ep(\gamma)\gamma_*$
under the identification 
$V_{\rho^*}\otimes V_{\rho}={\rm Hom}(V,V)_{\rho}$.
Then $\gamma(1\otimes G_{p,q}\circ \ep(\gamma)\gamma_*){\bf 1}\in C_1(M;V_{\rho^*}\otimes V_{\rho})$ can be described as
$\gamma(G_{p,q}\circ\ep(\gamma)\gamma_*)\in C_1(M;{\rm Hom}(V,V)_{\rho})$.
\begin{defini}
\begin{eqnarray*}
&&I_{\circ}(M,\rho;\grad f,G,c_f)\\
&&\hskip10mm=-c_f{\bf 1}-\sum_{\substack{p,q\in\Crit(f);\\{\rm ind}(p)={\rm ind}(q)+1}}
\sum_{\gamma\in{\mathcal M}(p,q)}\gamma(1\otimes G_{p,q}\circ \ep(\gamma)\gamma_*){\bf 1}_p \in C_1(M;V_{\rho^*}\otimes V_{\rho})\\
&&\hskip10mm\left(=-c_f{\bf 1}-\sum_{\substack{p,q\in\Crit(f);\\{\rm ind}(p)={\rm ind}(q)+1}}
\sum_{\gamma\in{\mathcal M}(p,q)}\gamma(G_{p,q}\circ \ep(\gamma)\gamma_*) \in C_1(M;{\rm Hom}(V,V)_{\rho})\right).
\end{eqnarray*}
\end{defini}
\begin{prop}[{\cite[Proposition~7.8.]{Shimizu}}]
For a closed oriented manifold $M$,
the 1-chain 

$I_{\circ}(M,\rho;\grad f,G,c_f)$ 
is a cycle and it represents $d(M,\rho)$.
\end{prop}
Even if $\partial M\not=\emptyset$, 
$I_{\circ}(M,\rho;\grad f,G,c_f)$ is still a cycle. Furthermore, its homology class modulo $H_1(M;\ZZ)$ is independent from choices of $f$, $\grad f$, $G$ and $c$.
Let $N$ be a closed oriented 3-manifold.
We assume that each boundary component is diffeomorphic to a torus $T^2=S^1\times S^1$.
Let $\rho:\pi_1(N)\to {\rm GL}(V)$ be an acyclic representation 
such that the restriction $\rho|_{\partial N}$ is also acyclic. 
We take two types of Morse functions.
Let $f_+,f_-:(N,\partial N)\to \RR$ be Morse functions satisfying
$f_+|_{[0,1]\times \partial N}(t,x)=t$ and $f_-|_{[0,1]\times \partial N}(t,x)=-t$ for any $t\in [0,1]$ and $x\in \partial N$. Here $[0,1]\times \partial N\subset N$ is a collar neighborhood of $\partial N$ with $\{0\}\times \partial N=\partial N$.  
We take a combinatorial propagator $G$ and a 1-chain $c\in C_1(N;\ZZ)$ in the same manner of $M$ with $\partial M=\emptyset$.
\begin{prop}\label{welldef of d}
Both
$$[I_{\circ}(N,\rho;\grad f_-,G,c_f)]\in H_1(N;V_{\rho^*}\otimes V_{\rho})/H_1(N;\ZZ)$$
and
$$[I_{\circ}(N,\rho;\grad f_+,G,c_f)]\in H_1(N;V_{\rho^*}\otimes V_{\rho})/H_1(N;\ZZ)$$
are independent of the choices of $f$, $\grad f$, $G$ and a 1-chain $c_f$.
\end{prop}
The proof of this proposition is given in Section~\ref{proof:welldef of d}.

From the above proposition, $d(N,\rho)$ and $d(N,\partial,\rho)$ can be defined as a homology class of $I_{\circ}(N,\rho;\grad f_-,c_f)$
and $I_{\circ}(N,\rho;\grad f_+,c_f)$ modulo $H_1(N;\ZZ)$ respectively, for a manifold $N$ with boundaries.
\begin{defini}
For a compact manifold $N$ with torus boundaries,
$$d(N,\rho)=[I_{\circ}(N,\rho;\grad f_-,G,c_f)]\in H_1(N;V_{\rho^*}\otimes V_{\rho})/H_1(N;\ZZ),$$
$$d(N,\partial N, \rho)=[I_{\circ}(N,\rho;\grad f_+,G,c_f)]\in H_1(N;V_{\rho^*}\otimes V_{\rho})/H_1(N;\ZZ).$$
\end{defini}
\begin{remark}
The homology group of the Morse-Smale complex of $f_-$
is $H_*(N)$ and that of $f_+$ is $H_*(N,\partial N)$.
Thus we adopt the above definition.
We show that $d(N,\partial N,\rho)=d(N,\rho)$ in Lemma~\ref{lemma:torus}. 
The next example plays an important role in the proof of $d(N,\partial N,\rho)=d(N,\rho)$.
\end{remark}
\subsection{Example}\label{example}
Let $N=[-1,1]\times T^2=[-1,1]\times (\RR/\ZZ)\times (\RR/\ZZ)$.
Let $\rho:\pi_1(N)\to {\rm GL}(V)$ be an acyclic representation. 
Let $f_0:(N,\partial N)\to (\RR,0)$ be the Morse function given by
$$f_0(t,x_1,x_2)=\left(\cos\frac{\pi}{2}t\right)(r_1+\cos2\pi x_1)(r_2+\cos 2\pi x_2)$$
for $(t,x_1)\in [-1,1]\times \RR/\ZZ$.
Here $r_1,r_2$ are positive constants satisfying $r_1\gg r_2\gg 1$.
It is easily checked that there are four critical points of $f$ and all of these are in $\{0\}\times T^2$. 
We denote by ${\rm Crit}(f_0)=\{{\rm NP},p,q,{\rm SP}\}$, where
${\rm ind}({\rm NP})=3,\ind(p)=\ind(q)=2,\ind({\rm SP})=1$. 
For a suitable choice of a gradient like vector field $\grad f_0$, 
there are two trajectories between 
${\rm NP}$ and $p$, $q$ and ${\rm SP}$ respectively.
Set 
${\mathcal M}({\rm NP},p)
=\{\gamma_{{\rm NP}\to p}^+,\gamma_{{\rm NP}\to p}^-\},$ 
$\mathcal M(q,{\rm SP})=\{\gamma_{q\to {\rm SP}}^+,\gamma_{q\to {\rm SP}}^-\}$
such that the 1-cycle
$(\gamma_{{\rm NP}\to p}^+)^{-1}+\gamma_{{\rm NP}\to p}^-$
is homologous to the 1-cycle
$(\gamma_{q\to{\rm SP}}^+)^{-1}+\gamma_{q\to{\rm SP}}^-$
in $H_1(\{0\}\times T^2;\ZZ)$. We denote by
$$l=[(\gamma_{{\rm NP}\to p}^+)^{-1}+ \gamma_{{\rm NP}\to p}^-]=
[(\gamma_{q\to {\rm SP}}^+)^{-1}+ \gamma_{q\to{\rm SP}}^-]
\in H_1(\{0\}\times T^2;\ZZ).$$

\includegraphics[width=12cm]{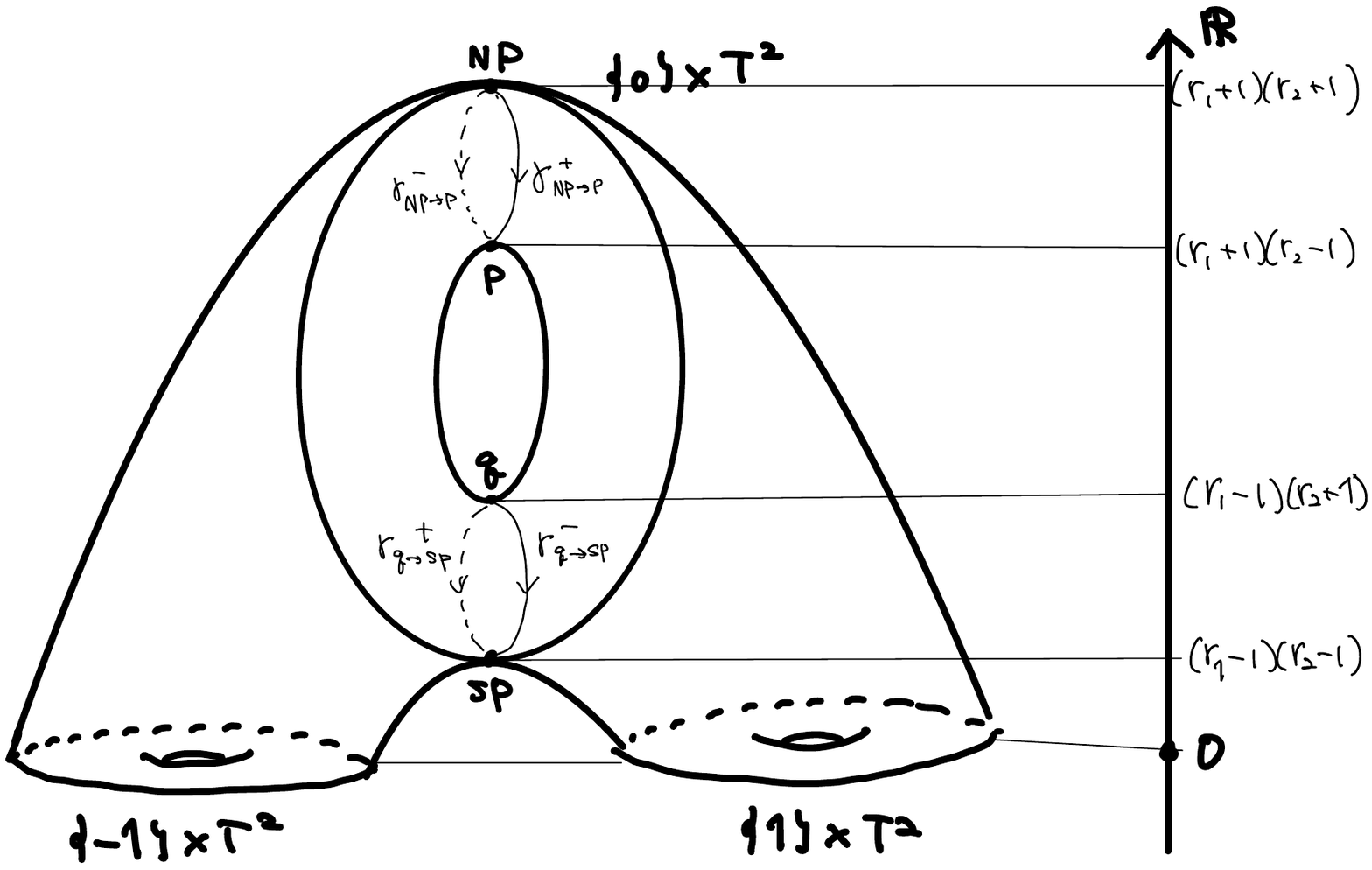}\label{pict1}

By fixing identifications $V_{\rm NP}\cong V$ and $V_q\cong V$,
we identify $V\cong V_{\rm NP}\cong V_p$ and $V \cong V_q\cong V_{\rm SP}$ via $(\gamma_{{\rm NP}\to p}^+)_*$ and $(\gamma_{q\to {\rm SP}}^+)_*$ respectively.
Under these identifications, $(\gamma_{{\rm NP}\to p}^-)_*$ can be given as an element $a=(\gamma_{{\rm NP}\to p}^-)_*\in {\rm GL}(V)$.
Since $[(\gamma_{{\rm NP}\to p}^+)^{-1}+ \gamma_{{\rm NP}\to p}^-]=
[(\gamma_{q\to {\rm SP}}^+)^{-1}+\gamma_{q\to{\rm SP}}^-]$, 
we have $(\gamma_{q\to {\rm SP}}^-)_*=a$.
We choose orientations of the ascending manifold $\mathcal A_p$ of $p$ and the descending manifold $\mathcal D_q$ so that 
$\ep(\gamma_{{\rm NP}\to p}^+)=1$ and $\ep(\gamma_{q\to {\rm SP}}^+)=1$.
Then $\ep(\gamma_{{\rm NP}\to p}^-)=-1$, $\ep(\gamma_{q\to {\rm SP}}^-)=-1$.
Therefore
$$\partial_{p,{\rm NP}}=(-1)^{{\rm ind}({\rm NP})}\ep(\gamma_{{\rm NP}\to p}^+)(\gamma_{{\rm NP}\to p}^+)_*+(-1)^{{\rm ind}({\rm NP})}\ep(\gamma_{{\rm NP}\to p}^-)(\gamma_{{\rm NP}\to p}^-)_*=-1+a,$$ $$\partial_{{\rm SP},q}=(-1)^{{\rm ind}(q)}\ep(\gamma_{q\to {\rm SP}}^+)(\gamma_{q\to {\rm SP}}^+)_*+(-1)^{{\rm ind}(q)}\ep(\gamma_{q\to {\rm SP}}^-)(\gamma_{q\to {\rm SP}}^-)_*=1-a.$$

Set $c_{f_0}=\gamma_{{\rm NP}\to p}^+-\gamma_{q\to {\rm SP}}^+\in C_1(N;\ZZ)$. Then $\partial c_{f_0}={\rm Crit}(f_0)$.
We take a combinatorial propagator $G$ as
$$G_3=G_{{\rm NP},p}\oplus G_{{\rm NP},q},\hskip2mm 
G_{{\rm NP},p}=(-1+a)^{-1}, G_{{\rm NP},q}=0,$$
$$G_2=G_{p,{\rm SP}}\oplus G_{q,{\rm SP}},\hskip2mm 
G_{p,{\rm SP}}=0, G_{q,{\rm SP}}=(1-a)^{-1}.$$
$\xymatrix{
 & V_{{\rm NP}} \ar[ddl]_(0.5){\gamma_{{\rm NP}\to p}^+}
\ar@<0.5ex>[ddl]^(0.5){\gamma_{{\rm NP}\to p}^-}
\ar@{.>}[ddr] 
& \\
&&\\
 V_p \ar@{.>}[ddr] 
& & 
V_q \ar[ddl]_(0.5){\gamma_{q\to {\rm NP}}^+}
\ar@<0.5ex>[ddl]^(0.5){\gamma_{q\to {\rm NP}}^-}
\\
&&\\
& V_{\rm SP} &
}$
\hskip10mm$\xymatrix{
 & V_{{\rm NP}} \ar[ddl]_{\substack{\partial_{p,{\rm NP}}\\=-1+a}}
\ar[ddr]^{\partial_{q,{\rm NP}}}
& \\
&&\\
 V_p \ar[ddr]_{\partial_{{\rm SP},p}} & & 
V_q \ar[ddl]^{\substack{\partial_{{\rm SP},q}\\ =1-a }}\\
&&\\
& V_{\rm SP} &
}$
\hskip10mm
$\xymatrix{
 & V_{{\rm NP}} & \\
&&\\
 V_p \ar[uur]^{\substack{G_{{\rm NP},p}\\=(-1+a)^{-1}}}  & & 
V_q \ar[uul]_{\substack{G_{{\rm NP},q}\\=0}}
\\
&&\\
& V_{\rm SP} \ar[uul]^{\substack{G_{p,{\rm SP}}\\=0}} 
\ar[uur]_{\substack{G_{q,{\rm SP}}\\=(1-a)^{-1}}}&
}$
\vskip2mm
{\small \hskip13mm trajectories
\hskip25mm boundary homomorphisms \hskip15mm combinatorial propagator}
\vskip5mm

\begin{table}[hbtp]
  \centering
  \begin{tabular}{lcr}
    \hline
 $\gamma$  & $\ep(\gamma)$  & $\gamma_*$ \\
    \hline \hline
   $\gamma_{{\rm NP}\to p}^+$   & 1& 1 \\
    $\gamma_{{\rm NP}\to p}^-$   & -1& $a$ \\
    $\gamma_{q\to {\rm SP}}^+$   & 1& 1 \\
    $\gamma_{q\to {\rm SP}}^-$   & -1& $a$ \\
     \hline 
  \end{tabular}
\vskip2mm
  \caption{$\ep(\gamma)$ and $\gamma_*$ of trajectories}
  \label{table:data_type}
\end{table}
\vskip5mm
Therefore $I_{\circ}(N,\rho;\grad f_0,G,c_{f_0})\in C_1(N;{\rm Hom}(V,V)_{\rho})$ is given as
\begin{eqnarray*}
&&I_{\circ}(N,\rho;\grad f_0,G,c_f)\\
&=&-c_{f_0}{\bf 1}-\sum_{\substack{r,s\in\Crit(f);\\{\rm ind}(r)={\rm ind}(s)+1}}
\sum_{\gamma\in{\mathcal M}(r,s)}\gamma(G_{r,s}\circ \ep(\gamma)\gamma_*)\\
&=&-(\gamma_{{\rm NP}\to p}^+-\gamma_{q\to {\rm SP}}^+){\bf 1}
-\gamma_{{\rm NP}\to p}^+(-1+a)^{-1}
-\gamma_{{\rm NP}\to p}^-(-1+a)^{-1}(-a)\\
&&
\hskip30mm-\gamma_{q\to {\rm SP}}^+(1-a)^{-1}
-\gamma_{q\to {\rm SP}}^-(1-a)^{-1}(-a)\\
&=&
-\gamma_{{\rm NP}\to p}^+(1+(-1+a)^{-1})
+\gamma_{{\rm NP}\to p}^-(-1+a)^{-1}a
-\gamma_{q\to {\rm SP}}^+(-1+(1-a)^{-1})
+\gamma_{q\to {\rm SP}}^-(1-a)^{-1}a
\end{eqnarray*}

Since $1+(-1+a)^{-1}=(-1+a)^{-1}\{(-1+a)+1\}=(-1+a)^{-1}a$ and 
$-1+(1-a)^{-1}=(1-a)^{-1}(-(1-a)+1)=(1-a)^{-1}a=-(-1+a)^{-1}a$, we have
\begin{eqnarray*}
&&I_{\circ}(N,\rho;\grad f_0,G,c_f)\\
&=&
-\gamma_{{\rm NP}\to p}^+(1+(-1+a)^{-1})
+\gamma_{{\rm NP}\to p}^-(-1+a)^{-1}a
-\gamma_{q\to {\rm SP}}^+(-1+(1-a)^{-1})
+\gamma_{q\to {\rm SP}}^-(1-a)^{-1}a\\
&=& (-\gamma_{{\rm NP}\to p}^++\gamma_{{\rm NP}\to p}^-)
(-1+a)^{-1}a-
(-\gamma_{q\to {\rm SP}}^++\gamma_{q\to {\rm SP}}^-)
(-1+a)^{-1}a.\\
\end{eqnarray*}

Thus
\begin{eqnarray*}
&&d(N,\partial N,\rho)\\
&=&[I_{\circ}(N,\rho;\grad f_0,G,c_{f_0})]\\
&=&[(-\gamma_{{\rm NP}\to p}^++\gamma_{{\rm NP}\to p}^-)
(-1+a)^{-1}a-
(-\gamma_{q\to {\rm SP}}^++\gamma_{q\to {\rm SP}}^-)
(-1+a)^{-1}a]\\
&=&
[(\gamma_{{\rm NP}\to p}^+)^{-1}+\gamma_{{\rm NP}\to p}^-]
(-1+a)^{-1}a-
[(\gamma_{q\to {\rm SP}}^+)^{-1}+\gamma_{q\to {\rm SP}}^-]
(-1+a)^{-1}a\\
&=&l(-1+a)^{-1}a-l(-1+a)^{-1}a\\
&=&0.
\end{eqnarray*}
Then we have $d([-1,1]\times T^2,\{-1,1\}\times T^2,\rho)=0$.

\section{Gluing formula for $d(M,\rho)$}
Let $N_a,N_b$ be compact oriented 3-manifolds with $\partial N_a\cong -\partial N_b$, where $-\partial N_b$ is $\partial N_b$ with the opposite orientation. 
We assume that each component of $\partial N_a(=\partial N_b)$ is diffeomorphic to a torus $T^2$.
Let $M$ be a closed oriented 3-manifold obtained by gluing $N_a$ with $N_b$ along the boundary:
$$M=N_a\cup N_b.$$
Let $\rho:\pi_1(M)\to {\rm GL}(V)$ be a representation such that all of  $\rho$, $\rho_a=\rho|_{N_a}, \rho_b=\rho|_{N_b}$ and $\rho|_{\partial N_a}(=\rho|_{\partial N_b})$ are acyclic.
\begin{example}
\begin{itemize}
\item[(1)] Let $N_a=D^2\times S^1, N_b=-D^2\times S^1$.
The boundary $\partial N_a=(\partial D^2)\times S^1$ is identified with $\partial N_b=(\partial D^2)\times S^1$ via the identity map.
Set $M=N_a\cup N_b\cong S^2\times S^1$.
Then $\pi_1(M)=\pi_1(N_a)=\pi_1(N_b)=\langle t\rangle\cong \ZZ$. 
In this case, for any non trivial representation $\rho:\langle t\rangle \ni t\mapsto \rho(t)\in U(1)\subset {\rm GL}(\CC)$, the representations $\rho$, $\rho_a,\rho_b$ and $\rho|_{\partial N_a}$ are acyclic.  
\item[(2)] Let $M$ be a 3-manifold obtained from $S^3$ by a Dehn  surgery along a knot $K\subset S^3$. In this case $N_a$ is a knot exterior $E(K)$ of $K$ and $N_b=D^2\times S^1$.
There are many representations $\rho: \pi_1(M)\to SL(2;\mathbb C)$ such that both $\rho$ and the restriction of $\rho$ to $\partial E(K)$ are acyclic (for example, see \cite{Kitano8} for the figure eight knot). For such representations, $\rho_a=\rho|_{N_a}$ and $\rho_b=\rho|_{N_b}$ are also acyclic.    
\item[(3)] 
Generally, the splice of knots, namely manifolds obtained by gluing two knot exteriors, give many examples of $M=N_a\cup N_b$. 
Y. Nozaki and the first author in  \cite{Kitano-Nozaki} studied on the Reidemeister torsion of the splice of knots.
Thus it is important to investigate the behavior of the defect $d$ under splice sum. 
How to treat a Morse-Smale complex for a given knot exterior, it is a next problem.
\end{itemize}
\end{example}
Take a Morse function $f:M\to \RR$ satisfying 
$$f|_{(-1,1)\times \partial N_a}(t,x)=t$$ for any $(t,x)\in (-1,1)\times\partial N_a$. 
Here $(-1,1)\times \partial N_a$ is identified with a tubular neighborhood of $\partial N_a\subset M$ such that $[0,1)\times \partial N_a\subset N_a$ and 
$(-1,0]\times \partial N_a\subset N_b$.

\includegraphics[width=8cm]{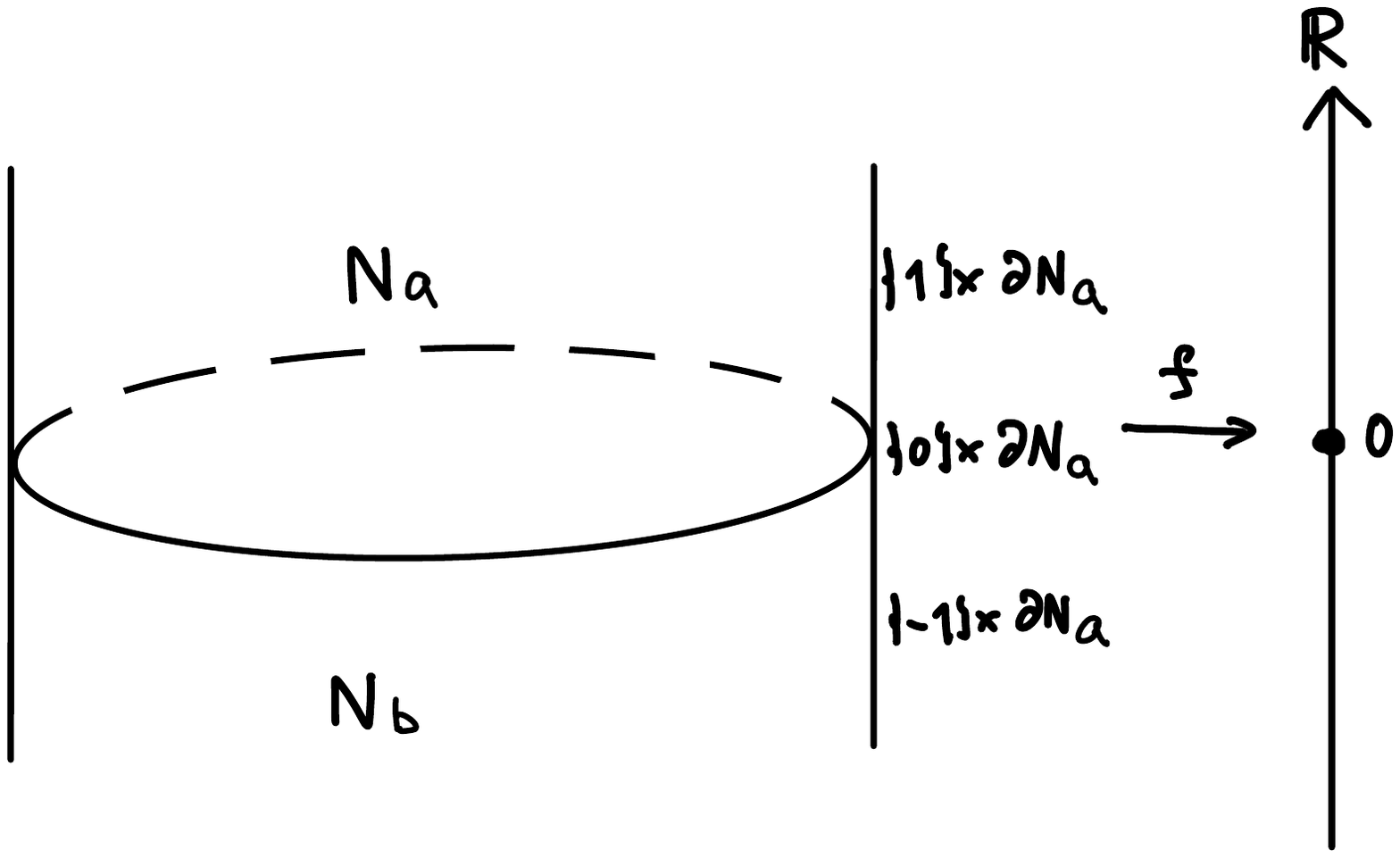}\label{pict1}

Take a gradient like vector field $\grad f$ satisfying the Morse-Smale conditions.
We assume that $\grad f$ is ``standard'' on the boundary $\partial N_a(=\partial N_b)$:
$$\grad f|_{\{0\}\times \partial N}=\frac{\partial }{\partial t}.$$
Write simply $f_a=f|_{N_a}, f_b=f|_{N_b}$, $\grad f_a=\grad f|_{N_a}$ and 
$\grad f_b=\grad f|_{N_b}$.
Let $(C_*^u,\partial_*^u)$ denote the Morse-Smale complex of $\grad f_u$ for $u=a,b$.
Since $\rho_a, \rho|_{\partial N_a}$ are acyclic, $H_*(N_a,\partial N_a;\rho_a)=H_*(C_*^a)=0$. Then $C_*^a$ is acyclic.
Take combinatorial propagators $G^{a},G^{b}$ for $(C_*^a,\partial_*^a),(C_*^b,\partial_*^b)$ respectively.
Take a 1-chain $c_u\in C_1(N_u;\ZZ)$
satisfying $\partial c_u=\widehat{\Crit}(f_u)$ for $u=a,b$.
We assume that the support of $c_u$ does not intersect $\partial N_u$ for $u=a,b$. 
Then the 1-chain $c_a+c_b\in C_1(M;\ZZ)$ satisfies $\partial (c_a+c_b)={\rm Crit}(f)$.
\begin{prop}\label{mainprop}
\begin{itemize}
\item[(1)] $I_{\circ}(N_a,\rho_a;\grad f_a,G^a,c_a)$ is a cycle of $C_1(N_a;V_{\rho_a^*}\otimes V_{\rho_a})$.
\item[(2)] $I_{\circ}(N_b,\rho_b;\grad f_b,G^b,c_b)$) is a cycle of
$C_1(N_b;V_{\rho_b^*}\otimes V_{\rho_b})$.
\item[(3)] $d(N_a\cup N_b,\rho)=[I_{\circ}(N_a,\rho_a;\grad f_a,G^a,c_a)+I_{\circ}(N_b,\rho_b;\grad f_b,G^b,c_b)]$.
\end{itemize}
\end{prop}
A proof of Proposition~\ref{mainprop} is given in Section~\ref{sec6}.

Proposition~\ref{mainprop} implies the following gluing formula.
\begin{them}\label{maintheorem}
$$d(N_a\cup N_b,\rho)=(\iota_a)_*d(N_a,\partial N_a,\rho_a)+(\iota_b)_*d(N_b,\rho_b).$$
Here $(\iota_u)_*:H_1(N_u;V_{\rho_u^*}\otimes V_{\rho_u})/H_1(N_u;\ZZ)\to
H_1(N_a\cup N_b; V_{\rho^*}\otimes V_{\rho})
/H_1(N_a\cup N_b;\ZZ)$ is the homomorphism induced by the inclusion 
$\iota_u:N_u\hookrightarrow N_a\cup N_b$ for $u=a,b$.
\end{them}
We show that $d(N_a,\partial N_a,\rho_a)=d(N_a,\rho_a)$ in Lemma~\ref{lemma:torus}.
Thus the following formula also holds.
\begin{them}\label{maintheorem2}
$$d(N_a\cup N_b,\rho)=(\iota_a)_*d(N_a,\rho_a)+(\iota_b)_*d(N_b,\rho_b).$$
\end{them}

\section{Proof of Proposition~\ref{mainprop}}\label{sec6}
Let $(C_*,\partial_*)$ denote the Morse-Smale complex of $\grad f$ on $M$.
Here $C_*$ can be written as follows. 
As a vector space,
$$C_i=C_i^a\oplus C_i^b$$
for any integer $i$. 
Let $\pi^{u}:C_*=C_*^a\oplus C_*^b\to C_*^u$ be the projection for $u=a,b$.
Set 
$$\partial_*^{\UL}=\pi^b\circ \partial_*|_{C_*^a}~~, ~~
\partial_*^{\LU}=\pi^a\circ \partial_*|_{C_*^b}.$$
To align the notation, we denote by $\partial_*^{\UU}=\partial^a_*,
\partial_*^{\LL}=\partial^b$ in this section.
Since there are no trajectories from $N_b$ to $N_a$, $\partial^{\LU}_*=0$.
Then $\partial_*$ can be decomposed as
$$\partial_*=\partial_*^{\UU}\oplus \partial_*^{\LL}\oplus \partial_*^{\UL}.$$

The following lemma plays an important role in the proof of Proposition~\ref{mainprop}.
\begin{lemma}\label{keylemma}
There is a family of homomorphisms $G_*^{\UL}=\{G_k^{\UL}:C_{k-1}^a\to C_k^b\}_{k} $ such that $G^M_*=G_*^{a}\oplus G_*^{b}\oplus G_*^{\UL}$ is a propagator for $C_*(M)$.
\end{lemma}
A proof of Lemma~\ref{keylemma} is given in the next section.

Let $\mathcal M$ denote the set of trajectories of $-\grad f$.
For a trajectory $\gamma\in \mathcal M$, we denote by $s(\gamma),t(\gamma)\in M$ the initial point and the terminal point of $\gamma$ respectively. 
Note that there is no trajectory $\gamma$ satisfying $s(\gamma)\in N_b$ and $t(\gamma)\in N_a$.
Thus the set $\mathcal M$ is divided into three subsets:
$$ \mathcal M=\mathcal M^{\UU}\sqcup \mathcal M^{\LL} \sqcup \mathcal M^{\UL}.$$ 
Here 
\begin{itemize}
\item $\mathcal M^{\UU}=\{\gamma\mid s(\gamma),t(\gamma)\in N_a\}$,
\item $\mathcal M^{\LL}=\{\gamma\mid s(\gamma),t(\gamma)\in N_b\}$ and
\item 
$\mathcal M^{\UL}=\{\gamma\mid s(\gamma)\in N_a,t(\gamma)\in N_b\}$.
\end{itemize}

Let $G^{\overrightarrow{uv}}_*=\pi^v\circ G|_{C^u_*}:C_*^u\to C_*^v$ be the corresponding part of $G$ for $u,v=a,b$.
Then by the Lemma~\ref{keylemma}, $G_*^{\UU}=G_*^a,
G_*^{\LL}=G^b_*$ and $G_*^{\LU}=0$. Then we have
\begin{eqnarray*}
&&I_{\circ}(M,\rho;\grad f,G,c_a+c_b)\\
&=&-(c_a+c_b){\bf 1}-\sum_{\gamma\in \mathcal M}\gamma(1\otimes G_{s(\gamma),t(\gamma)}\circ \ep(\gamma)\gamma_*){\bf 1}_p\\
&=&-c_a{\bf 1}-\sum_{\gamma\in \mathcal M^{u}}\gamma(1\otimes G^{\UU}_{s(\gamma),t(\gamma)}\circ  \ep(\gamma)\gamma_*){\bf 1}_p
-c_b{\bf 1}-\sum_{\gamma\in \mathcal M^{l}}\gamma(1\otimes G^{\LL}_{s(\gamma),t(\gamma)}\circ  \ep(\gamma)\gamma_*){\bf 1}_p\\
&&-\sum_{\gamma\in \mathcal M^{\UL}}\gamma(1\otimes G^{\LU}_{s(\gamma),t(\gamma)}\circ  \ep(\gamma)\gamma_*){\bf 1}_p\\
&=&I_{\circ}(N_a,\rho_a;\grad f_a,G^{a},c_a)+I_{\circ}(N_b,\rho_b;\grad f_b,G^{b},c_b).
\end{eqnarray*}
Recall that $I_{\circ}(M,\rho;\grad f,G,c_a+c_b)$ is a cycle.
Since any trajectory in $\mathcal M^{a}$ does not intersect $\partial N_a$
and the 1-chain $c_a$ also does not intersect $\partial N_a$, the support of $I_{\circ}(N_a,\rho_a;\grad f_a,G^a,c_a)$ is in $N_a\setminus \partial N_a$. Similarly, the support of $I_{\circ}(N_b,\rho_b;\grad f_b,G^b,c_b)$ is in $N_b\setminus \partial N_b$.
Then $I_{\circ}(N_a,\rho_a;\grad f_a,G^a,c_a)$ does not intersect with 
$I_{\circ}(N_b,\rho_b;\grad f_b,G^b,c_b)$ in $M=N_a\cup N_b$.
Therefore the equality
$I_{\circ}(M,\rho;\grad f,G,c_a+c_b)=I_{\circ}(N_a,\rho_a;\grad f_a,G^a,c_a)+I_{\circ}(N_b,\rho_b;\grad f_b,G^b,c_b)$
implies that both $I_{\circ}(N_a,\rho_a;\grad f_a;G^a,c_a)$ and $I_{\circ}(N_b,\rho^l;\grad f_b;G^b,c_b)$ are cycles.
\subsection{Proof of Lemma~\ref{keylemma}}
Since $G_*^a=G^{\UU}_*,G_*^b=G_*^{\LL}$ are propagators, for each $k$, the required condition 
$\partial_{k+1}\circ G^M_{k+1}+G_k^M\circ \partial_{k}={\rm id}_{C_k(M)}$ for $G_k^M$
is equivalent with 
$$ \partial_{k+1}^{\UL}\circ G_{k+1}^{\UU}+\partial_{k+1}^{\LL}\circ G_{k+1}^{\UL}
+G_{k}^{\UL}\circ \partial_{k}^{\UU}+G_{k}^{\LL}\circ \partial_{k}^{\UL}
=0:C_{k}^a\to C_{k}^b.$$ 
We prove the existence of homomorphisms $\{G_k^{\UL}:C_{k-1}^a\to C_k^b\}$ satisfying the above conditions by induction on $k$.
$$\xymatrix{
 C_{k+1} \ar[d]^(0.5){\partial_{k+1}}& = &C_{k+1}^a\ar[d]^(0.5){\partial_{k+1}^{\UU}}\ar[drr]|(0.65){\partial_{k+1}^{\UL}}& \oplus 
& C_{k+1}^b \ar[d]_{\partial^{\LL}_{k+1}}\\
 C_{k} \ar[d]^(0.5){\partial_{k}} \ar@<0.5ex>@{.>}[u]^{G_{k+1}^M}
& = &C_{k}^a\ar@<0.5ex>@{.>}[u]^{G_{k+1}^{\UU}}\ar@{.>}[urr]|(0.35){G_{k+1}^{\UL}}
\ar[d]^(0.5){\partial_{k}^{\UU}}\ar[drr]|(0.65){\partial_{k}^{\UL}}
& \oplus 
& C_{k}^l\ar[d]_{\partial^{\LL}_{k}} \ar@<-0.5ex>@{.>}[u]_{G^{\LL}_{k+1}}\\
C_{k-1} \ar@<0.5ex>@{.>}[u]^{G_{k}^M} 
& = & C_{k-1}^a
\ar@<0.5ex>@{.>}[u]^{G_{k}^{\UU}}\ar@{.>}[urr]|(0.35){G_{k}^{\UL}}
\ar@<0.5ex>@{.>}[u]^{G^{\UU}_{k}}&
\oplus
 &C_{k-1}^b\ar@<-0.5ex>@{.>}[u]_{G^{\LL}_{k}}
}$$

We first prove that there exists a homomorphism $G^{\UL}_1:C_0^{a}\to C_1^{b}$ satisfying the condition
\begin{eqnarray}
\label{13}
\partial^{\UL}_1 \circ G^{\UU}_1+\partial_1^{\LL}\circ G_1^{\UL}+
G_0^{\UL}\circ \partial^{\UU}_0+G^{\LL}_0\circ \partial _0^{\UL} &=&0:C_0^a\to C_0^b
\end{eqnarray}
Since $C_{-1}^a=C_{-1}^b=\{0\}$, the equality (\ref{13}) is equivalent with 
\begin{eqnarray}\label{13b}
\partial^{\UL}_1\circ G^{\UU}_1+\partial^{\LL}_1\circ G_1^{\UL}=0:C_0^a\to C_0^b
\end{eqnarray}
Set $$G_1^{\UL}=-G_1^{\LL}\circ \partial_1^{\UL}\circ G_1^{\UU}.$$
\begin{remark}
The following two figures are a diagrammatic description of equality~(\ref{13b}) and the definition of $G_1^{\UL}$ respectively.
\begin{eqnarray*}
\left[\vcenter{\hbox{\xymatrix{
C_{1}^a\ar[dr] & C_{1}^b \\
C_{0}^a\ar@<0.5ex>@{.>}[u] & C_{0}^b
}}}\right]
+
\left[\vcenter{\hbox{\xymatrix{
C_{1}^a & C_{1}^b \ar[d]\\
C_{0}^a\ar@{.>}[ur]& C_{0}^b
}}}\right]
=0,
\hskip5mm
\left[\vcenter{\hbox{\xymatrix{
C_{1}^a & C_{1}^b\\
C_{0}^a\ar@{.>}[ur]|{G_1^{\UL}}& C_{0}^b 
}}}\right]=
-
\left[\vcenter{\hbox{\xymatrix{
C_{1}^a\ar[dr] & C_{1}^b\\
C_{0}^a\ar@<0.5ex>@{.>}[u]& C_{0}^b \ar@<-0.5ex>@{.>}[u]
}}}\right]
.\end{eqnarray*}

\end{remark}  
\vskip10mm
Since $G^{\LL}_*$ is a propagator and $C_{-1}^b=\{0\}$, we have
$$\partial^{\LL}_1\circ G^{\LL}_1={\rm id}_{C_0^b}.$$
Then 
\begin{eqnarray*}
&&\partial^{\UL}_1 \circ G^{\UU}_1+\partial_1^{\LL}\circ G_1^{\UL}\\
&=& \partial_1^{\UL}\circ G_1^{\UU}
-\partial^{\LL}_1\circ G^{\LL}_1 \circ \partial^{\UL}\circ G_1^{\UU}\\
&=&({\rm id}_{C_0^b}- \partial^{\LL}_1\circ G^{\LL}_1)\circ \partial_1^{\UL}\circ G_1^{\UU}
\\
&=&0.
\end{eqnarray*}

We next prove that there exists 
$G_{k+1}^{\UL}:C_k^a\to C_{k+1}^b$ satisfying 
\begin{eqnarray}\label{eq1}
\partial^{\UL}_{k+1}\circ G^{\UU}_{k+1}+\partial_{k+1}^{\LL}\circ G_{k+1}^{\UL}+
G_k^{\UL}\circ \partial^{\UU}_k+G^{\LL}_k\circ \partial _k^{\UL} &=&0:C_k^a\to C_{k}^b,
\end{eqnarray}
under the assumption that $G_1^{\UL},\cdots,G_{k}^{\UL}$ exist.

\begin{remark}
The following figure is a diagrammatic description of equality~(\ref{eq1}).

\begin{eqnarray*}
\left[\vcenter{\hbox{\xymatrix{
C_{k+1}^a\ar[dr] & C_{k+1}^b \\
C_{k}^a\ar@<0.5ex>@{.>}[u] & C_{k}^b\\
C_{k-1}^a&C_{k-1}^b
}}}\right]
+
\left[\vcenter{\hbox{\xymatrix{
C_{k+1}^a & C_{k+1}^b \ar[d]\\
C_{k}^a\ar@{.>}[ur]& C_{k}^b\\
C_{k-1}^a&C_{k-1}^b
}}}\right]
+
\left[\vcenter{\hbox{\xymatrix{
C_{k+1}^a & C_{k+1}^b \\
C_{k}^a\ar[d] & C_{k}^b\\
C_{k-1}^a\ar@{.>}[ur]  &C_{k-1}^b
}}}\right]
+
\left[\vcenter{\hbox{\xymatrix{
C_{k+1}^a& C_{k+1}^b \\
C_{k}^a\ar[dr]& C_{k}^b\\
C_{k-1}^a&C_{k-1}^b\ar@<-0.5ex>@{.>}[u]
}}}\right]
=0,
\end{eqnarray*}
\end{remark}  
\vskip10mm
Since $\partial_{k+1}^{\LL}\circ G_{k+1}^{\LL}+G_k^{\LL}\circ \partial_k^{\LL}= {\rm id}_{C_k^b}$, he equality (\ref{eq1}) is equivalent with 
\begin{eqnarray*}
(\partial_{k+1}^{\LL}\circ G_{k+1}^{\LL}+G_k^{\LL}\circ \partial_k^{\LL})\circ(\partial^{\UL}_{k+1}\circ  G^{\UU}_{k+1}+\partial_{k+1}^{\LL}\circ G_{k+1}^{\UL}+
G_k^{\UL}\circ \partial^{\UU}_k+G^{\LL}_k\circ \partial _k^{\UL})=0.
\end{eqnarray*}
Then it is enough to show that there exists a homomorphism $G^{\UL}_{k+1}$ satisfying both
$$
\partial_{k+1}^{\LL}\circ G_{k+1}^{\LL}\circ (\partial^{\UL}_{k+1}\circ  G^{\UU}_{k+1}+\partial_{k+1}^{\LL}\circ G_{k+1}^{\UL}+
G_k^{\UL}\circ \partial^{\UU}_k+G^{\LL}_k\circ \partial _k^{\UL})=0$$
and
$$
G_k^{\LL}\circ \partial_k^{\LL}\circ (\partial^{\UL}_{k+1}\circ  G^{\UU}_{k+1}+\partial_{k+1}^{\LL}\circ G_{k+1}^{\UL}+
G_k^{\UL}\circ \partial^{\UU}_k+G^{\LL}_k\circ \partial _k^{\UL})=0.$$
Set 
\begin{eqnarray}\label{eq2}
G^{\UL}_{k+1}=-\left(G_{k+1}^{\LL}\circ \partial_{k+1}^{\UL}\circ G_{k+1}^{\UU}
+G_{k+1}^{\LL}\circ G_k^{\UL}\circ\partial_{k}^{\UU}
+G_{k+1}^{\LL}\circ G_k^{\LL}\circ \partial_{k}^{\UL}\right).
\end{eqnarray}

\begin{remark}
The following figure is a diagrammatic description of the definition~(\ref{eq2}).

\begin{eqnarray*}
G_{k+1}^{\UL}=
-
\left[\vcenter{\hbox{\xymatrix{
C_{k+1}^a\ar[dr] & C_{k+1}^b\\
C_{k}^a\ar@<0.5ex>@{.>}[u]& C_{k}^b \ar@<-0.5ex>@{.>}[u]\\
C_{k-1}^a&C_{k-1}^b
}}}\right]
-
\left[
\vcenter{\hbox{\xymatrix{
C_{k+1}^a& C_{k+1}^b\\
C_{k}^a\ar[d]& C_{k}^b\ar@<-0.5ex>@{.>}[u]\\
C_{k-1}^a\ar@{.>}[ur]&C_{k-1}^b
}}}\right]
-
\left[
\vcenter{\hbox{\xymatrix{
C_{k+1}^a& C_{k+1}^b\\
C_{k}^a\ar[dr]& C_{k}^b\ar@<-0.5ex>@{.>}[u]\\
C_{k-1}^a&C_{k-1}^b\ar@<-0.5ex>@{.>}[u]
}}}\right]
.\end{eqnarray*}
\end{remark}  
\vskip10mm
\begin{lemma}
For the above $G_{k+1}^{\UL}$,
$$\partial_{k+1}^{\LL}\circ G_{k+1}^{\LL}\circ (\partial^{\UL}_{k+1}\circ  G^{\UU}_{k+1}+\partial_{k+1}^{\LL}\circ G_{k+1}^{\UL}+
G_k^{\UL}\circ \partial^{\UU}_k+G^{\LL}_k\circ \partial _k^{\UL})=0.$$
\end{lemma}
\begin{proof}
\begin{eqnarray*}
&&\partial_{k+1}^{\LL}\circ G_{k+1}^{\LL}\circ (\partial^{\UL}_{k+1}\circ  G^{\UU}_{k+1}+\partial_{k+1}^{\LL}\circ G_{k+1}^{\UL}+
G_k^{\UL}\circ \partial^{\UU}_k+G^{\LL}_k\circ \partial _k^{\UL})\\
&=&\partial_{k+1}^{\LL}\circ G_{k+1}^{\LL}\circ \partial^{\UL}_{k+1}\circ  G^{\UU}_{k+1}
+\partial_{k+1}^{\LL}\circ (G_{k+1}^{\LL}\circ\partial_{k+1}^{\LL})\circ G_{k+1}^{\UL}\\
&&\hskip15mm+\partial_{k+1}^{\LL}\circ G_{k+1}^{\LL}\circ G_k^{\UL}\circ \partial^{\UU}_k
+\partial_{k+1}^{\LL}\circ G_{k+1}^{\LL}\circ G^{\LL}_k\circ \partial _k^{\UL}\\
&=&\partial_{k+1}^{\LL}\circ G_{k+1}^{\LL}\circ \partial^{\UL}_{k+1}\circ  G^{\UU}_{k+1}
+\partial_{k+1}^{\LL}\circ ({\rm id}_{C_{k+1}^b}-\partial_{k+2}^{\LL}\circ G_{k+2}^{\LL})\circ G_{k+1}^{\UL}\\
&&\hskip15mm+\partial_{k+1}^{\LL}\circ G_{k+1}^{\LL}\circ G_k^{\UL}\circ \partial^{\UU}_k
+\partial_{k+1}^{\LL}\circ G_{k+1}^{\LL}\circ G^{\LL}_k\circ \partial _k^{\UL}\\
&=&\partial_{k+1}^{\LL}\circ G_{k+1}^{\LL}\circ \partial^{\UL}_{k+1}\circ  G^{\UU}_{k+1}
+\partial_{k+1}^{\LL}\circ G_{k+1}^{\UL}\\
&&\hskip15mm+\partial_{k+1}^{\LL}\circ G_{k+1}^{\LL}\circ G_k^{\UL}\circ \partial^{\UU}_k
+\partial_{k+1}^{\LL}\circ G_{k+1}^{\LL}\circ G^{\LL}_k\circ \partial _k^{\UL}\\
&=&\partial_{k+1}^{\LL}\circ (G_{k+1}^{\LL}\circ \partial^{\UL}_{k+1}\circ  G^{\UU}_{k+1}
+G_{k+1}^{\UL}+G_{k+1}^{\LL}\circ G_k^{\UL}\circ \partial^{\UU}_k+G_{k+1}^{\LL}\circ G^{\LL}_k\circ \partial _k^{\UL})\\
&=&0.
\end{eqnarray*}
\end{proof}
\begin{lemma}
The above ${G}_{k+1}^{\UL}:C_k^a\to C_{k+1}^b$ satisfies
$$
G_k^{\LL}\circ \partial_k^{\LL}\circ (\partial^{\UL}_{k+1}\circ  G^{\UU}_{k+1}+\partial_{k+1}^{\LL}\circ G_{k+1}^{\UL}+
G_k^{\UL}\circ \partial^{\UU}_k+G^{\LL}_k\circ \partial_k^{\UL})=0.$$
\end{lemma}
\begin{proof}
\begin{eqnarray*}
&&G_k^{\LL}\circ \partial_k^{\LL}\circ (\partial^{\UL}_{k+1}\circ  G^{\UU}_{k+1}+\partial_{k+1}^{\LL}\circ G_{k+1}^{\UL}+
G_k^{\UL}\circ \partial^{\UU}_k+G^{\LL}_k\circ \partial_k^{\UL})\\
&=&G_k^{\LL}\circ (\partial_k^{\LL}\circ \partial^{\UL}_{k+1})\circ  G^{\UU}_{k+1}+
G_k^{\LL}\circ \partial_k^{\LL}\circ G_k^{\UL}\circ \partial^{\UU}_k+G_k^{\LL}\circ \partial_k^{\LL}\circ G^{\LL}_k\circ \partial_k^{\UL}.
\end{eqnarray*}
Since $\partial_{k}\circ \partial_{k+1}=0$, 
we have $\partial_k^{\LL}\circ \partial^{\UL}_{k+1}=-\partial^{\UL}_{k}\circ \partial_{k+1}^{\UU}$. Then
\begin{eqnarray*}
&&G_k^{\LL}\circ (\partial_k^{\LL}\circ \partial^{\UL}_{k+1})\circ  G^{\UU}_{k+1}+G_k^{\LL}\circ \partial_k^{\LL}\circ G_k^{\UL}\circ \partial^{\UU}_k+G_k^{\LL}\circ \partial_k^{\LL}\circ G^{\LL}_k\circ \partial_k^{\UL}\\
&=&G_k^{\LL}\circ (-\partial_k^{\UL}\circ \partial^{\UU}_{k+1})\circ  G^{\UU}_{k+1}+G_k^{\LL}\circ \partial_k^{\LL}\circ G_k^{\UL}\circ \partial^{\UU}_k+G_k^{\LL}\circ \partial_k^{\LL}\circ G^{\LL}_k\circ \partial _k^{\UL}\\
&=&-G_k^{\LL}\circ \partial_k^{\UL}\circ (\partial^{\UU}_{k+1}\circ  G^{\UU}_{k+1})+G_k^{\LL}\circ \partial_k^{\LL}\circ G_k^{\UL}\circ \partial^{\UU}_k+G_k^{\LL}\circ \partial_k^{\LL}\circ G^{\LL}_k\circ \partial_k^{\UL}\\
&=&-G_k^{\LL}\circ \partial_k^{\UL}\circ ({\rm id}_{C_k^a}-G_k^{\UU}\circ \partial^{\UU}_{k})+G_k^{\LL}\circ \partial_k^{\LL}\circ G_k^{\UL}\circ \partial^{\UU}_k+G_k^{\LL}\circ \partial_k^{\LL}\circ G^{\LL}_k\circ \partial_k^{\UL}\\
&=&-G_k^{\LL}\circ \partial_k^{\UL}+G_k^{\LL}\circ \partial_k^{\UL}\circ G_k^{\UU}\circ \partial^{\UU}_{k}+G_k^{\LL}\circ \partial_k^{\LL}\circ G_k^{\UL}\circ \partial^{\UU}_k+G_k^{\LL}\circ \partial_k^{\LL}\circ G^{\LL}_k\circ \partial_k^{\UL}\\
&=&-G_k^{\LL}\circ \partial_k^{\UL}+G_k^{\LL}\circ (\partial_k^{\UL}\circ G_k^{\UU}+\partial_k^{\LL}\circ G_k^{\UL})\circ \partial^{\UU}_k+G_k^{\LL}\circ \partial_k^{\LL}\circ G^{\LL}_k\circ \partial_k^{\UL}.
\end{eqnarray*}
By the assumption of the induction, $G_k^{\UL}$ satisfying 
\begin{eqnarray*}
\partial^{\UL}_{k}\circ G^{\UU}_{k}+\partial_{k}^{\LL}\circ G_{k}^{\UL}+
G_{k-1}^{\UL}\circ \partial^{\UU}_{k-1}+G^{\LL}_{k-1}\circ \partial_{k-1}^{\UL} &=&0:C_{k-1}^a\to C_{k-1}^b.
\end{eqnarray*}
Thus we have
\begin{eqnarray*}
&&-G_k^{\LL}\circ \partial_k^{\UL}+G_k^{\LL}\circ (\partial_k^{\UL}\circ G_k^{\UU}+\partial_k^{\LL}\circ G_k^{\UL})\circ \partial^{\UU}_k+G_k^{\LL}\circ \partial_k^{\LL}\circ G^{\LL}_k\circ \partial_k^{\UL}\\
&=&-G_k^{\LL}\circ \partial_k^{\UL}+G_k^{\LL}\circ (-G_{k-1}^{\UL}\circ \partial^{\UU}_{k-1}-G^{\LL}_{k-1}\circ \partial_{k-1}^{\UL})\circ \partial^{\UU}_k+G_k^{\LL}\circ \partial_k^{\LL}\circ G^{\LL}_k\circ \partial_k^{\UL}\\
&=&-G_k^{\LL}\circ \partial_k^{\UL}-G_k^{\LL}\circ G^{\LL}_{k-1}\circ \partial_{k-1}^{\UL}\circ \partial^{\UU}_k+G_k^{\LL}\circ (\partial_k^{\LL}\circ G^{\LL}_k)\circ \partial_k^{\UL}\\
&=&-G_k^{\LL}\circ \partial_k^{\UL}-G_k^{\LL}\circ G^{\LL}_{k-1}\circ \partial_{k-1}^{\UL}\circ \partial^{\UU}_k+G_k^{\LL}\circ ({\rm id}_{C_{k-1}^{\LL}}-G^{\LL}_{k-1}\circ \partial_{k-1}^{\LL})\circ \partial _k^{\UL}\\
&=&-G_k^{\LL}\circ G^{\LL}_{k-1}\circ \partial_{k-1}^{\UL}\circ \partial^{\UU}_k-G_k^{\LL}\circ G^{\LL}_{k-1}\circ \partial_{k-1}^{\LL}\circ \partial_k^{\UL}\\
&=&-G_k^{\LL}\circ G^{\LL}_{k-1}\circ (\partial_{k-1}^{\UL}\circ \partial^{\UU}_k+
\partial_{k-1}^{\LL}\circ \partial_k^{\UL})\\
&=& -G_k^{\LL}\circ G^{\LL}_{k-1}\circ 0\\
&=&0.
\end{eqnarray*}
\end{proof}
\begin{remark}
The first two lines of the above proof imply that 
$$
G_k^{\LL}\circ \partial_k^{\LL}\circ (\partial^{\UL}_{k+1}\circ  G^{\UU}_{k+1}+\partial_{k+1}^{\LL}\circ \widetilde{\bf G}_{k+1}^{\UL}+
G_k^{\UL}\circ \partial^{\UU}_k+G^{\LL}_k\circ \partial _k^{\UL})=0$$
for any homomorphism $\widetilde{\bf G}_{k+1}^{\UL}:C_k^a\to C_{k+1}^b$.
\end{remark}

\section{Proof of Proposition~\ref{welldef of d}}\label{proof:welldef of d} 
Let us denote by $-N$ the manifold $N$ with the opposite orientation.
Take two collar neighborhoods as
$[0,1)\times \partial N\subset N$ with $\{0\}\times \partial N=\partial N$
and $(-1,0]\times \partial N\subset -N$ with $\{0\}\times \partial N=\partial(-N)$ respectively.

Let $f^+_1,f^+_2:(N,\partial N)\to (\RR,0)$ be Morse functions satisfying
$f_i^+|_{[0,1)\times \partial N}(t,x)=t$ for $i=1,2$.
Let $f^-_1,f_2^-:(-N,\partial(-N))\to (\RR,0)$ be Morse functions satisfying
$f_i^-|_{(-1,0]\times \partial N}(t,x)=t$ for $i=1,2$.
Take gradient like vector fields $\grad f^+_1,\grad f^+_2,\grad f^-_1$ and $\grad f^-_2$ satisfying
$\grad f_1^+|_{[0,1)\times \partial N}=\grad f_2^+|_{[0,1)\times \partial N}=\frac{\partial}{\partial t}$ and 
$\grad f^-_1|_{(-1,0]\times \partial N}=\grad f^-_2|_{(-1,0]\times \partial N}=\frac{\partial}{\partial t}$.

Let $D(N)=N\cup -N$ be the double of $N$.
Let $\pi:D(N)\to N$ be a continuous map given by $\pi(x)=x$ for $x\in N$ and $\pi(y)=y$ for $y\in -N$.
Obviously, a representation $\rho\circ \pi_*$ of $\pi_1(D(N))$ satisfying $\rho\circ\pi_*|_{N}=\rho, \rho\circ\pi_*|_{-N}=\rho$.
Since $\rho$ and $\rho|_{\partial N}$ are acyclic, $\rho\circ\pi_*$ is also acyclic.

There are two Morse functions 
$f_1^+\cup f^-_1,f_2^+\cup f^-_1:N\cup -N\to \RR$ and these gradient like vector fields $\grad f_1^+\cup \grad f^-_1, \grad f_2^+\cup \grad f^-_1$. 
Take propagators $G_1^+,G_2^+$, $G^-_1$ and $G^-_2$ for the Morse-Smale complex of $\grad f^+_1,\grad f^+_2$, $\grad f^-_1$ and $\grad f^-_2$ respectively.
Take 1-chains $c_1^+,c_2^+,c_1^-,c_2^-\in C_1(N;\ZZ)$
satisfying $\partial c_1^+={\rm Crit}(f_1^+)$,
$\partial c_2^+={\rm Crit}(f_2^+)$,
$\partial c_1^-={\rm Crit}(f_1^-)$ and
$\partial c_2^-={\rm Crit}(f_2^-)$.
From Proposition~\ref{mainprop}, we have
\begin{eqnarray*}
d(D(N),\rho\circ\pi_*)
&=&[I_{\circ}(N,\rho;\grad f^+_1,G_1^+,c_1^+)+I_{\circ}(-N,\rho;\grad f^-_1,G^-_1,c_1^-)]\\
&=&[I_{\circ}(N,\rho;\grad f^+_2,G_2^+,c_2^+)+I_{\circ}(-N,\rho;\grad f^-_1,G^-_1,c_1^-)].
\end{eqnarray*}
Then we have 
\begin{eqnarray*}
&&[I_{\circ}(N,\rho;\grad f^+_1,G_1^+,c_1^+)-I_{\circ}(N,\rho;\grad f^+_2,G_2^+,c_2^+)]\\
&&\hskip30mm=0\in H_1(D(N);V_{(\rho\circ\pi_*)^*}\otimes V_{\rho\circ\pi_*})/H_1(D(N);\ZZ).
\end{eqnarray*}
By using Mayer-Vietoris arguments, we can easily check that 
 $H_1(N;V_{\rho^*}\otimes V_{\rho})/H_1(N;\ZZ)\to 
H_1(D(N);V_{(\rho\circ\pi_*)^*}\otimes V_{\rho\circ\pi_*})/
H_1(D(N):\ZZ)$ is injection.
Then 
$$[I_{\circ}(N,\rho;\grad f^+_1,G_1^+,c_1^+)]=[I_{\circ}(N,\rho;\grad f^+_2,G_2^+,c_2^+)]\in H_1(N;V_{\rho^*}\otimes V_{\rho}).$$
Similarly, 
$$[I_{\circ}(N,\rho;\grad f^-_1,G_1^-,c_1^-)]=[I_{\circ}(N,\rho;\grad f_2^-,G_2^-,c_2^-)].$$
Furthermore, the following lemma holds.
\begin{lemma}\label{lemma:torus}
$d(N,\partial N,\rho)=d(N,\rho)$.
\end{lemma}
\begin{proof}[Proof of Lemma~\ref{lemma:torus}]
Without loss of generality, we assume that $N$ has only one boundary component: $\partial N=T^2$.
Let $f_0:([-1,1]\times \partial N,\{-1,1\}\times \partial N)\to (\RR,0)$ be  
the Morse function introduced in the example given in Section~\ref{example}. From the example, we have
$$d([-1,1]\times \partial N,\{-1,1\}\times \partial N,\rho)=0.$$
Let $N$ be the manifold obtained  by gluing $N$ and $[-1,1]\times \partial N$ along $N\supset \partial N$ and 
${-1}\times \partial N\subset [-1,1]\times \partial N$. Obviously, $N'\cong N$. 
The Morse functions $f^-_1$ on  $N$ and $f^0$ on $[-1,1]\times \partial N$ give a Morse function on $N'$.

\includegraphics[width=8cm]{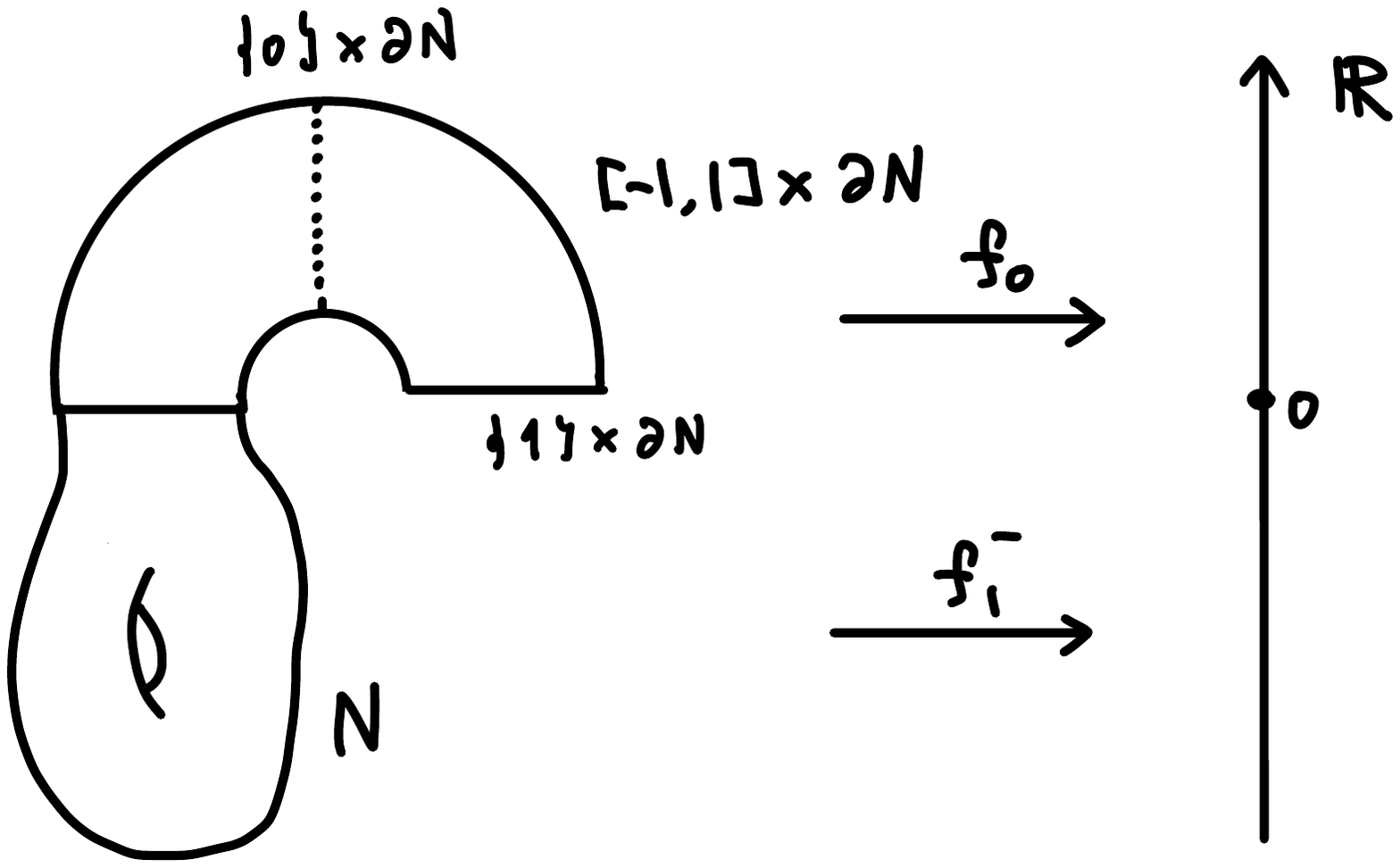}\label{pict2}

The Morse function $f^-_1\cup f_0$ gives $d(N',\partial N',\rho)$. Then we have
\begin{eqnarray*}
d(N,\partial N,\rho)&=&d(N',\partial N',\rho)\\
&=&d(N,\rho)+d([-1,1]\times \partial N,\{-1,1\}\times \partial N, \rho)\\
&=&d(N,\rho).
\end{eqnarray*}
\end{proof}


\end{document}